\documentclass[11pt]{amsart}

\usepackage{latexsym, amssymb,enumerate}
\usepackage{appendix}

\newcommand{\be}{\begin{equation}}
\newcommand{\ee}{\end{equation}}
\newcommand{\beq}{\begin{eqnarray}}
\newcommand{\eeq}{\end{eqnarray}}
\newcommand{\cM}{\mathcal{M}}

\def\R{{\mathfrak R}}

\newtheorem{prop}{Proposition}[section]
\newtheorem{theo}[prop]{Theorem}

\newtheorem{coro}[prop]{Corollary}
\newtheorem{rema}[prop]{Remark}
\newtheorem{exam}[prop]{Example}

\def\begeq{\begin{equation}}
\def\endeq{\end{equation}}
\def\p{\partial}

\def\R{\mathbb R}

\def\tr{{\rm tr}}

\def\d{\delta}

\def\s{\sigma}

\def \ds{\displaystyle}
\def \vs{\vspace*{0.1cm}}

\def\odot{\setbox0=\hbox{$\bigcirc$}\relax \mathbin {\hbox
to0pt{\raise.5pt\hbox to\wd0{\hfil $\wedge$\hfil}\hss}\box0 }}

\numberwithin{equation} {section}

\begin{document}

\title[ Gauss-Bonnet-Chern mass] {The Gauss-Bonnet-Chern mass of conformally flat manifolds}

\author{Yuxin Ge}
\address{Laboratoire d'Analyse et de Math\'ematiques Appliqu\'ees,
CNRS UMR 8050,
D\'epartement de Math\'ematiques,
Universit\'e Paris Est-Cr\'eteil Val de Marne, \\61 avenue du G\'en\'eral de Gaulle,
94010 Cr\'eteil Cedex, France}
\email{ge@u-pec.fr}
\author{Guofang Wang}
\address{ Albert-Ludwigs-Universit\"at Freiburg,
Mathematisches Institut
Eckerstr. 1
D-79104 Freiburg}
\email{guofang.wang@math.uni-freiburg.de}

\author{Jie Wu}
\address{School of Mathematical Sciences, University of Science and Technology
of China Hefei 230026, P. R. China
\and
 Albert-Ludwigs-Universit\"at Freiburg,
Mathematisches Institut
Eckerstr. 1
D-79104 Freiburg
}
\email{jie.wu@math.uni-freiburg.de}

\thanks{The project is partly supported by SFB/TR71
``Geometric partial differential equations''  of DFG}

\begin{abstract}

In this paper we show  positive mass theorems and  Penrose type inequalities for the Gauss-Bonnet-Chern mass, which was introduced recently  in \cite{GWW},
for    asymptotically flat    CF manifolds and its rigidity.

\end{abstract}
\maketitle

\section{Introduction}

Recently motivated by the  Einstein-Gauss-Bonnet theory \cite{BD, Wheeler} and the pure Lovelock theory \cite{Lo, CTZ}, we introduced in \cite{GWW} (and \cite{GWW2})  the Gauss-Bonnet-Chern mass by using the Gauss-Bonnet curvature
 \begin{equation}\label{Lk}
L_k:=\frac{1}{2^k}\d^{i_1i_2\cdots i_{2k-1}i_{2k}}
_{j_1j_2\cdots j_{2k-1}j_{2k}}{R_{i_1i_2}}^{j_1j_2}\cdots
{R_{i_{2k-1}i_{2k}}}^{j_{2k-1}j_{2k}}.
\end{equation}
When $k=1$, $L_1$ is just the scalar curvature $R$. When $k=2$, it is the (second) so-called the Gauss-Bonnet curvature
\[L_2 = R_{\mu\nu\rho\s}R^{\mu\nu\rho\s}-4R_{\mu\nu}R^{\mu\nu}+R^2,\]
which appeared at the first time in the paper of Lanczos \cite{Lan} in 1938. For general $k$ it is the Euler integrand in the
Gauss-Bonnet-Chern theorem \cite{Chern1, Chern2} if $n=2k$ and
 is therefore called the dimensional continued Euler density
in physics if $k<n/2$. Here $n$ is the dimension. In this paper we are interested in the case $k<n/2$. The Gauss-Bonnet-Chern mass introduced in \cite{GWW} is defined
\begin{equation}\label{GBC}
m_k=m_{GBC}=c(n,k)
\lim_{r\to\infty}\int_{S_r}P_{(k)}^{ijlm}\partial_m g_{jl} \nu_{i}dS,\end{equation}
with
\[c(n,k)=\frac{(n-2k)!} {2^{k-1}(n-1)!\omega_{n-1}},\]
where $\omega_{n-1}$ is the volume of $(n-1)$-dimensional standard unit sphere and  $S_r$ is the Euclidean coordinate sphere, $dS$ is the volume element on $S_r$ induced by the Euclidean metric, $\nu$ is the outward unit normal to $S_r$ in $\mathbb{R}^n$ and the derivative is the ordinary partial derivative.
Here the tensor $P_{(k)}$ is decided by the decomposition
\begin{equation}\label{decompose}
L_k=P_{(k)}^{ijlm}R_{ijlm}.
\end{equation}
In this paper we use the Einstein summation convention. The tensor $P_{(k)}$ has a crucial property of divergence-free, which guarantees the Gauss-Bonnet-Chern mass is well-defined
and is a geometric invariant, under a suitable decay condition.  See Section 2 below or \cite{GWW}.
When $k=1$, $$P_{(1)}^{ijlm}=\frac 12 (g^{il}g^{jm}-g^{im}g^{jl}),$$ and $m_1$
is just the ADM mass introduced 
by Arnowitt,  Deser, and Misner \cite{ADM}
for asymptotically flat Riemannian manifolds. For a similar mass see also \cite{NL}.

A complete  manifold $(\cM^n,g)$ is said to be an asymptotically flat (AF) of order $\tau$ (with one end) if there is a compact $K$ such that
$\cM\setminus K$ is diffeomorphic to  $\mathbb{R}^n\setminus B_R(0)$ for some $R>0$ and in the standard coordinates in $\mathbb{R}^n$, the metric $g$ has the following expansion
 $$g_{ij}=\delta_{ij}+\sigma_{ij},$$
 with $$|\sigma_{ij}|+ r|\partial\sigma_{ij}|+ r^2|\partial^2\sigma_{ij}|=O(r^{-\tau}),$$
 where $r$ and $\partial$ denote the Euclidean distance and the standard derivative operator on $\mathbb{R}^n$ respectively.
The condition for the welldefinedness of the  Gauss-Bonnet-Chern mass is
\begin{equation}\label{con}
\tau> \frac{n-2k}{k+1},\end{equation}
and $L_k$ is integrable over $\cM$. In this case, the Gauss-Bonnet-Chern mass is a geometric invariant,
which is a generalization of the work of Bartnik for the ADM mass $m_1$ \cite{Bar}.

The positive mass theorem for  the ADM mass $m_{ADM}=m_1$, which plays an important role in differential geometry,
 was proved by Schoen and Yau \cite{SchoenYau1} for $3\le n\le 7$ and by Witten for general spin manifolds. See aslo \cite{Loh1, Loh}. Its refinement, the Penrose inequality,  was proved by Huisken-Ilmanen \cite{HI} and Bray \cite{BP} for $n=3$ and Bray-Lee \cite{BL} for $n\le 7$. Recently there
are many interesting works on special, but interesting classes of asymptotically flat manifolds.
In \cite{Lam} Lam  showed the positive mass theorem
and the Penrose inequality for asymptotically flat graphs in $\R^{n+1}$ by using an elementary, but elegant proof. See also the generalizations
of Lam's work in \cite{dLG1,dLG2,HW1, HW2}. The Penrose type inequality is proved for conformally flat manifolds by Freire-Schwartz \cite{FSw},
Jauregui \cite{J} and Schwartz \cite{Sw} by using the relation between mass and the capacity. This relation was used already in the proof of Penrose inequality in \cite{BP}. For this relation,
see also \cite{BI} and \cite{BM}. It is interesting to see that there is a  deep relation between the AMD mass and various geometric objects.

We are interested in generalizing the above results to our  Gauss-Bonnet-Chern $m_{GBC}=m_k$ ($k\geq 2$). Motivated by the work of Lam \cite{Lam}, we showed a positive mass theorem and an optimal
Penrose inequality when $\cM$ is an asymptotically flat graphs in $\R^{n+1}$ in \cite{GWW}.
This Penrose inequality
establishes a relationship between the mass $m_{GBC}$ and more geometric objects \cite{GWW}.
In this paper we are interested in studying $m_{GBC}$ mass on conformally flat manifolds.

A {\it conformally flat} manifold with or without boundary, CF manifold for short,
is a manifold $(\cM^n, g)=(\R^n\slash \Omega, e^{-2u}\delta)$,
 where $\delta$ is the canonical Euclidean metric on $\mathbb{R}^n$, $\Omega$ is a smooth bounded
 (possibly empty, not necessarily connected) open set and $u$ is  smooth.
A CF manifold $(\cM^n,g)$  is called an   asymptotically flat  CF manifold of decay order $\tau$ if
\begin{equation}\label{decay}
|u|+|x||\nabla u|+|x|^2|\nabla^2 u|=O(|x|^{-\tau}).
\end{equation}
In this paper we  always assume that $k< \frac n2$, $\tau>\frac{n-2k}{k+1}$ and $L_k$ is integrable.

\

First we have a positive mass theorem. 
\begin{theo}
\label{mainthm1}
Let $(\cM^n,g)=(\mathbb{R}^n, e^{-2u}\delta)$ be an asymptotically flat CF manifold.
 Assume  further that $L_j(g)\ge 0$ for any $j\le k$.
 Then the mass $m_{GBC} \ge 0$. Moreover,  equality holds if and only if $u\equiv 0$, i.e., $\cM$ is the Euclidean space.
\end{theo}

The condition  $L_j(g)\ge 0$ for any $j\le k$
 here is equivalent to  $g\in \Gamma_k$, which will be discussed in  Section 2 below. A similar result was announced by Li-Nguyen
in \cite{NL}.

For the Gauss-Bonnet-Chern mass, $m_{2j+1} $ has different behavior with $m_{2j} $. The former behaves like the ADM mass $m_1$ and the latter
like  $m_2$. For $k$ even, we  have also a  positive mass theorem for metrics   in a non-positive cone.

\begin{theo}
\label{mainthm2}
Let  $k$ be even and $(\cM^n,g)=(\mathbb{R}^n, e^{-2u}\delta)$ be an asymptotically flat  CF manifold. 
Assume  $(-1)^jL_j\ge 0$ for any $j\le k$. Then the mass $m_{GBC} \ge 0$.
 Moreover, equality holds if and only if $u\equiv 0$, i.e., $\cM$ is the Euclidean space.
\end{theo}

 \

 Theorem \ref{mainthm1}  and Theorem \ref{mainthm2}  provide a support for our conjecture on the positivity of the Gauss-Bonnet-Chern mass in \cite{GWW}.
 Furthermore, from our proof we have a Penrose type inequality.

 \begin{theo}\label{mainthm3}
Let $ (\cM^n,g)= (\mathbb{R}^n\setminus \Omega, e^{-2u}\delta)$ be an asymptotically flat CF manifold. 
Assume  that  
$\Omega$ is convex,
 $\partial \cM=(\Omega, e^{-2u}\delta)$ is a horizon of $(\cM,g)$ and $u$  is  constant on $\partial\Omega$. Assume  further that $L_j(g)\ge 0$ for any $j\le k$. Then we have Penrose type inequalities
\beq
\label{penrose_n}
\begin{array}{llll}
m_k
&\ge&\ds\left(\frac{|\p\Omega|}{\omega_{n-1}}\right)^{\frac{n-2k}{n-1}}.
\end{array}
\eeq
Moreover, if $k\geq 2$, we have the following strengthened Penrose type inequality
\beq
\label{penrose1_n}
\begin{array}{llll}
m_k
&\ge&\ds \left(\frac{\int_{\p\Omega}R}{(n-1)(n-2)\omega_{n-1}}\right)^{\frac{n-2k}{n-3}},\
\end{array}
\eeq
where $R$ is the scalar curvature of $\p\Omega$  as a hypersurface in $\R^n$.
\end{theo}

The assumptions on the boundary $\partial \Omega$ can be reduced by the result of Guan-Li \cite{GL} and the results could be slightly strengthened.
For more details see Section 4 below. Unlike the Penrose inequality obtained in \cite{GWW}, this  Penrose inequality is not optimal. Our Penrose inequality
is motivated by the work of Jauregui in \cite{J}, who obtained (\ref{penrose_n}) for $k=1$. The idea is to express the mass via
various integral identities.



The rest of the paper is organized as follows. In Section 2 we recall the definitions of  the Gauss-Bonnet curvature $L_k$ and the $\s_k$-scalar curvature and their relationship when
the underlying manifolds are locally conformally flat. In Section 3 we prove the positive mass theorems, Theorem \ref{mainthm1} and  Theorem \ref{mainthm2}.
Theorem \ref{mainthm3} is proved in Section 4.

\section{The Gauss-Bonnet curvatures and the $\s_k$-scalar curvatures}
 We recall the definition of generalized  {\it $k$-th Gauss-Bonnet curvature}
 \begin{equation}
L_k:=\frac{1}{2^k}\d^{i_1i_2\cdots i_{2k-1}i_{2k}}
_{j_1j_2\cdots j_{2k-1}j_{2k}}{R_{i_1i_2}}^{j_1j_2}\cdots
{R_{i_{2k-1}i_{2k}}}^{j_{2k-1}j_{2k}}.
\end{equation}
Here the generalized Kronecker delta is defined by
\[
 \d^{j_1j_2 \dots j_r}_{i_1i_2, \dots i_r}=\det\left(
\begin{array}{cccc}
\d^{j_1}_{i_1} & \d^{j_2}_{i_1} &\cdots &  \d^{j_r}_{i_1}\\
\d^{j_1}_{i_2} & \d^{j_2}_{i_2} &\cdots &  \d^{j_r}_{i_2}\\
\vdots & \vdots & \vdots & \vdots \\
\d^{j_1}_{i_r} & \d^{j_2}_{i_r} &\cdots &  \d^{j_r}_{i_r}
\end{array}
\right).
\]

When $k=2$, we can write
\begin{equation}\label{s1}\begin{array}{rcl}
 L_2 &=& \ds\vs R_{\mu\nu\rho\s}R^{\mu\nu\rho\s}-4R_{\mu\nu}R^{\mu\nu}+R^2\\
&=&\ds \vs|W|^2+\frac {n-3}{n-2}\bigg(\frac{n}{n-1}R^2-4|Ric|^2\bigg)\\
&=&\ds |W|^2+8(n-2)(n-3)\s_2(A_g)\\
&=&\ds R_{ijkl}P_{(2)}^{ijkl},
\end{array}
\end{equation}
where
\begin{equation}\label{P}
P_{(2)}^{ijkl}=R^{ijkl}+R^{jk}g^{il}-R^{jl}g^{ik}-R^{ik}g^{jl}+R^{il}g^{jk}+\frac12R(g^{ik}g^{jl}-g^{il}g^{jk}),
\end{equation}
$W$ denotes the Weyl tensor, $Ric$ the Ricci tensor, $R$ the scalar curvature and
$$
A_g:=\frac{1}{n-2}\left( Ric-\frac{R}{2(n-1)}g\right),
$$
 the Schouten tensor.
$P_{(2)}$ is the divergence-free part of the Riemann curvature  tensor $Riem$.
For the general $L_k$-curvature, the corresponding $P_{(k)}$ curvature is
\begin{equation}\label{Pk}
P_{(k)}^{st lm}:=\frac{1}{2^k}\d^{i_1i_2\cdots i_{2k-3}i_{2k-2}st}
_{j_1j_2\cdots  j_{2k-3}j_{2k-2} j_{2k-1}j_{2k}}{R_{i_1i_2}}^{j_1j_2}\cdots
{R_{i_{2k-3}i_{2k-2}}}^{j_{2k-3}j_{2k-2}}g^{j_{2k-1}l}g^{j_{2k}m}.
\end{equation}
Recall that $L_k=P_{(k)}^{ijlm}R_{ijlm}$ and the tensor $P_{(k)}$ has the following crucial property.
\begin{prop} \label{property}
The tensor $P_{(k)}$ has the same symmetry and anti-symmetry as the Riemann curvature tensor and satisfies
\[\nabla_iP_{(k)}^{ijlm}=0.\]
\end{prop}
\begin{proof}
The case $k=1$ is trivial. We have proved the $k=2$ case in \cite{GWW}. For the general case,
it follows from the symmetry of the Riemann curvature tensor and the differential Bianchi identity. We skip the proof here.
\end{proof}

Now we consider the case that $(\cM^n,g)$ is a   conformally flat manifold of dimension $n\ge 5$. Namely,
 $(\cM^n,g)=(\mathbb{R}^n, e^{-2u}\delta)$, where $\delta$ is the canonical Euclidean metric on $\mathbb{R}^n$. In this case, the curvature
 $L_k$ is just the $\s_k$-scalar curvature (up to a multiple constant), which was considered  by Viaclovsky in \cite{Via1} and has been intensively studied
 in the $\s_k$ Yamabe problem.

For the convenience of the reader,
we recall some basic properties on the elementary symmetric functions (see for example \cite{Guan,ChHang, Via1} ). For $1\le k\le n$ and $\lambda=(\lambda_1,\cdots,\lambda_n)\in \R^n$, the $k$-th elementary symmetric function is defined as
\begin{eqnarray*}
\s_k(\lambda):=\sum_{i_1<i_2<\cdots i_k} \lambda_{i_1}\cdots\lambda_{i_k}.
\end{eqnarray*}
The definition can be extended to symmetric matrices. For a symmetric matrix $B$, denote $\lambda(B)=(\lambda_1(B),\cdots,\lambda_n(B))$ be the eigenvalues of $B$. We set
\[
\s_k(B):=\s_k(\lambda(B)).
\]
We define also $\s_0(B)=1$. Let $I$ be the identity matrix. Then we have for any $t\in\R$,
\[
\s_n(I+tB)=\det(I+tB)=\sum_{i=0}^{n}\s_i(B)t^i.
\]
 We recall the definition of the Garding cone: for $1\le k\le n$, let $\Gamma_k^+$ (resp. ${\Gamma_k}$) is a cone in $\R^n$ determined by
\[
\Gamma_k^+=\{\lambda\in\R^n:\quad \s_1(\lambda)>0,\cdots,\s_k(\lambda)>0\}.
\]
\[
(\mbox{resp. }\Gamma_k=\{\lambda\in\R^n:\quad \s_1(\lambda)\ge 0,\cdots,\s_k(\lambda)\ge0\}).
\]
A symmetric matrix $B$ is called belong to $\Gamma_k^+$ (resp. $\Gamma_k$) if $\lambda(B)\in \Gamma_k^+$ (resp. $\lambda(B)\in  \Gamma_k$).
The $k$-th Newton transformation is defined as follows
\begin{equation}\label{Newtondef}
(T_k)^{i}_{j}(B):=\frac{\partial \s_{k+1}}{\partial b^{i}_{j}}(B),
\end{equation}
where $B=(b^{i}_{j})$. If there is no confusion, we omit the index $k$. We recall some basic properties about $\s_k$ and $T$.\\
\begin{eqnarray}\label{sigmak}
\s_k(B)& =&\ds\frac{1}{k!}\d^{i_1\cdots i_k}
_{j_1\cdots j_k}b_{i_1}^{j_1}\cdots
{b_{i_k}^ {j_k}}=\frac{1}{k} \tr(T_{k-1}B),
\\
(T_k)^{i}_j(B) & =& \ds \frac{1}{k!}\d^{ii_1\cdots i_{k}}
_{j j_1\cdots j_{k}}b_{i_1}^{j_1}\cdots
{b_{i_{k}}^{j_{k}}}\label{Tk}\\
& =& \ds \sum_{i=0}^{k}\s_{k-i}(B)(-B)^i=\s_{k}(B)I-\s_{k-1}(B)B+\cdots+(-1)^kB^k.\nonumber
\end{eqnarray}
It is well-known that $\sigma_k^{1/k}$ is  concave in $\Gamma_k$, which implies that
\begin{equation}\label{eq_a1}
\s_k(A+B)\ge \s_k(A)+\s_k(B), \quad \hbox{ for  any } A,B\in \Gamma_k.
\end{equation}
The $\sigma_k$-scalar curvature $\sigma_k(g)$ is defined in \cite{Via1} by
$$\sigma_k(g):=\sigma_k(g^{-1}A_g),$$
where $A_g$ is the Schouten tensor of $g$.
\begin{prop}
Let $(\cM^n,g)$ be a locally conformally flat metric of dimension $n$. Assume $2k<n$. Then
\begin{equation}\label{relation}
L_k=2^k k! \frac{(n-k)!}{(n-2k)!} \s_k(g).
\end{equation}
\end{prop}
\proof
We recall the decomposition of  the Riemann curvature tensor
$$
Riem=W+A\odot g.
$$
As $W\equiv 0$, we have
\begin{equation}\label{Rm}
{R_{i_1i_2}}^{j_1j_2}={A_{i_1}}^{j_1}{\d_{i_2}}^{j_2}+{\d_{i_1}}^{j_1}{A_{i_2}}^{j_2}-{A_{i_1}}^{j_2}{\d_{i_2}}^{j_1}-{\d_{i_1}}^{j_2}{A_{i_2}}^{j_1}.
\end{equation}
It follows that
$$
\begin{array}{lllc}
L_k&=&2^k \d^{i_1i_2\cdots i_{2k-1}i_{2k}}
_{j_1j_2\cdots j_{2k-1}j_{2k}}{A_{i_1}}^{j_1}{\d_{i_2}}^{j_2}\cdots
{A_{i_{2k-1}}}^{j_{2k-1}}{\d_{i_{2k}}}^{j_{2k}}\\
&=&2^k (n-k)\cdots(n-2k+1) \d^{i_1i_3\cdots i_{2k-1}}
_{j_1j_3\cdots j_{2k-1}}{A_{i_1}}^{j_1}\cdots
{A_{i_{2k-1}}}^{j_{2k-1}}\\
&=&2^k k! (n-k)\cdots(n-2k+1)\s_k(A).
\end{array}
$$
Here we use the facts $$
\begin{array}{rcl}
\d^{i_1i_2\cdots i_{2k-1}i_{2k}}
_{j_1j_2\cdots j_{2k-1}j_{2k}}{A_{i_1}}^{j_1}{\d_{i_2}}^{j_2} &=& \ds\vs  \d^{i_1i_2\cdots i_{2k-1}i_{2k}}
_{j_1j_2\cdots j_{2k-1}j_{2k}}{\d_{i_1}}^{j_1}{A_{i_2}}^{j_2} \\&=&-\d^{i_1i_2\cdots i_{2k-1}i_{2k}}
_{j_1j_2\cdots j_{2k-1}j_{2k}}{A_{i_1}}^{j_2}{\d_{i_2}}^{j_1}=- \d^{i_1i_2\cdots i_{2k-1}i_{2k}}
_{j_1j_2\cdots j_{2k-1}j_{2k}}{\d_{i_1}}^{j_2}{A_{i_2}}^{j_1}. \end{array}$$  \qed 

\

For $k=\frac n2$ see \cite{Via1}. Another important property will be the following.

\begin{prop}\label{pro2}  (see \cite{Via1})
Let $(\cM^n,g)$ be a locally conformally flat manifold of dimension $n$. Then $T_{k-1}(A)$ is divergence-free.
\end{prop}

Without the conformal flatness Proposition \ref{pro2} still  holds  for $k=2$, i.e., $T_1$ is divergence-free,  which was proved in \cite{Via1}.

\section{Positive Mass Theorem for  CF manifolds and Rigidity}

In this section we prove Theorem \ref{mainthm1} and  Theorem \ref{mainthm2}.
For the proof  we need one more well-known property.
\begin{prop}
\label{prop1.1}
Let $u:\R^n\to\R$ be some smooth function. Denote $D^2 u=(u_{ij})$ be the hessian matrix of $u$ with respect to Euclidean metric. Then $T_k(D^2 u)$ is divergence-free, that is,
\[
\partial_i T_k^{ij}(D^2 u)=\partial_j T_k^{ij}(D^2 u)=0.
\]
\end{prop}

\begin{rema}
Note that in Proposition \ref{prop1.1} the divergence-free is with respect to the standard euclidean metric $\delta$
and  in Proposition \ref{pro2} the divergence-free is with respect to the metric $g=e^{-2u}\delta$.
\end{rema}

For an asymptotically flat CF manifold, we first have an equivalent form of Gauss-Bonnet-Chen mass defined by (\ref{GBC}). By (\ref{decompose}), (\ref{Rm}) together with Proposition \ref{property}, we have
$$
L_k=4P^{ijlm}_{(k)}A_{il}g_{jm}=-4P^{ijjl}_{(k)}A_{il}e^{-2u}.$$
On the other hand,  from (\ref{sigmak}) and (\ref{relation}) we have
$$L_k=2^k (k-1)! \frac{(n-k)!}{(n-2k)!}(T_{k-1}(A))^{il} A_{il}.$$
 For the  Gauss-Bonnet-Chern mass (\ref{GBC}) we have
\begin{eqnarray*}
m_k&:=&\frac{(n-2k)!} {2^{k-1}(n-1)!\;\omega_{n-1}}
\lim_{r\to\infty}\int_{S_r}P_{(k)}^{ijlm}\partial_m g_{jl} \nu_{i}dS\\
&=&\frac{(n-2k)!} {2^{k-1}(n-1)!\;\omega_{n-1}}
\lim_{r\to\infty}\int_{S_r}-2e^{-2u}P_{(k)}^{ijjl} u_l \nu_{i}dS.
\end{eqnarray*}
Combining all together, we thus obtain the following equivalent form of (\ref{GBC}),
\begin{equation}\label{equv.form}
m_k=\lim_{r\to\infty}\frac{(k-1)! (n-k)!}{(n-1)!\; \omega_{n-1}}\int_{S_r}(T_{k-1}(A))^{ij}u_j\nu_i d S.
\end{equation}
This formula would be useful in the computation of the Gauss-Bonnet-Chern mass.
Now we start to prove Theorem \ref{mainthm1}.

\

{\noindent{\it Proof of Theorem \ref{mainthm1}.}
 Since $g=e^{-2u}\delta,$ a direct computation gives 
\begin{eqnarray*}
Ric&=&(n-2)(D^2 u+\frac{1}{n-2}(\Delta u)\delta+du\otimes du-|\nabla u|^2),\\
R&=&e^{2u}(2(n-1)\Delta u-(n-1)(n-2)|\nabla u|^2),
\end{eqnarray*}
which imply
\begin{equation}\label{schouten}
A_g:=\frac{1}{n-2}\left( Ric-\frac{Rg}{2(n-1)}\right)=D^2u-\frac{|\nabla u|^2}{2}I+du\otimes du.
\end{equation}
Here $\nabla$ and $\Delta$ are operators with respect to the Euclidean metric $\delta$  and $D^2$ are the Hessian operator.
Since
\[
T_{k-1}(D^2 u)=T_{k-1}(A)+O(|x|^{-k\tau-2k+2}),
\] which  follows from (\ref{decay}) and (\ref{Tk}), we have by (\ref{equv.form})
\begin{eqnarray}\label{eq_a}
m_k&=&\lim_{r\to\infty}\frac{(k-1)! (n-k)!}{(n-1)!\; \omega_{n-1}}\int_{S_r}(T_{k-1}(D^2 u))^{ij}u_j\nu_i d S.
\end{eqnarray}
 Applying Proposition \ref{prop1.1} and Green's formula, we obtain
\begin{eqnarray}
\label{Gr}
\int_{S_r}(T_{k-1}(D^2 u))^{ij}u_j\nu_i d S=\int_{B_r}(T_{k-1}(D^2 u))^{ij}u_{ij} dx=k\int_{B_r}\s_k(D^2 u)dx.
\end{eqnarray}
Now, we write
\[
D^2u=A+\frac{|\nabla u|^2}{2}I-du\otimes du.
\]
It is crucial to see  that the matrix $\frac{|\nabla u|^2}{2}I-du\otimes du$ has one eigenvalue $-\frac{|\nabla u|^2}{2}$ and $n-1$ eigenvalues $\frac{|\nabla u|^2}{2}$. Therefore, $B:=\frac{|\nabla u|^2}{2}I-du\otimes du\in  \Gamma^+_k$ for $k<n/2$, for
\[ \sigma_j(B)=\frac{(n-1)!(n-2j)}{2^jj!(n-j)!} |\nabla u|^{2j} \quad \hbox {for any } j\le k<n/2.\]
It follows from (\ref{eq_a1}) that
\begin{eqnarray}
\label{expan}
\s_k(D^2 u)&=&\s_k(A+B)\\
&\ge&\ds\s_k(A)+\s_k(B)=\s_k(A)+\frac{(n-1)!(n-2k)}{2^kk!(n-k)!} |\nabla u|^{2k}.
\end{eqnarray}
Finally, we infer
\begin{eqnarray}
\label{massexpan}
m_k
\ge &\ds&\frac{(n-2k)!}{2^k(n-1)!\omega_{n-1}} \int_{\cM} e^{(n-2k)u}L_k(g)dvol_g\nonumber\\
&&\ds+ \frac{n-2k}{2^k}\int_{\cM} e^{(n-2k)u}|\nabla u|_g^{2k}dvol_g.  
\end{eqnarray}
This yields the positivity of the mass $m_k$. Moreover, if $m_k=0$, we have $\nabla u\equiv 0$. Hence $u\equiv 0$, that is, $g$ is the Euclidean metric.
 We finish the proof of the Theorem.
\qed

\begin{rema}
In the above proof, the calculations before (\ref{massexpan}) are with respect to the Euclidean metric $\delta$, namely $\sigma_k(A)$ means $\sigma_k(\delta^{-1}A)$. Hence from (\ref{relation}) that $L_k=2^k k! \frac{(n-k)!}{(n-2k)!} e^{2ku}\s_k(A),$
which has be used in (\ref{massexpan}).
\end{rema}

\

{\noindent{\it Proof of Theorem \ref{mainthm2}.} Let $v:=e^u$. Thus, the conformal metric is  written as $g=v^{-2}\delta$. For such a representation of the metric, the Schouten tensor (\ref{schouten}) can be written as
 \[
A=\frac{D^2v}{v}-\frac{|\nabla v|^2\delta}{2v^2}.
\]
Let $\alpha\in\mathbb {R}$ be some sufficiently negative number to be fixed later.
As in the proof of Theorem \ref{mainthm1}, it follows from the decay condition (\ref{decay})
 of $u$ that
 \[
(T_{k-1}(D^2 u))^{ij}u_j\nu_i =v^{\alpha}(T_{k-1}(D^2 v))^{ij}v_j\nu_i+O(|x|^{-(k+1)\tau-2k+1}) ,
\]
which  implies  from  (\ref{Tk}) and  (\ref{equv.form})
\begin{eqnarray}\label{eq_b}
m_k&=&\lim_{r\to\infty}\frac{(k-1)! (n-k)!}{(n-1)!\; \omega_{n-1}}\int_{S_r}v^{\alpha}(T_{k-1}(D^2 v))^{ij}v_j\nu_i d S.
\end{eqnarray}
Thus, a direct calculation leads to
\begin{eqnarray*}
m_k&=&\ds \lim_{r\to\infty}\frac{ (k-1)! (n-k)!}{(n-1)!\; \omega_{n-1}}\int_{S_r}v^{\alpha}(T_{k-1}(D^2 v))^{ij}v_j\nu_i d S\\
&=&\ds \frac{ (k-1)! (n-k)!}{(n-1)!\; \omega_{n-1}}\int_{\mathbb{R}^n} v^{\alpha}(T_{k-1}(D^2 v))^{ij}v_{ji} d x\\
&&+\ds \frac{ (k-1)! (n-k)!}{(n-1)!\; \omega_{n-1}}\int_{\mathbb{R}^n} v^{\alpha}{(T_{k-1}(D^2 v))^{ij}}_{,i}v_{j} d x\\
&&+\ds \frac{ (k-1)! (n-k)!\alpha }{(n-1)!\; \omega_{n-1}}\int_{\mathbb{R}^n} v^{\alpha-1}(T_{k-1}(D^2 v))^{ij}v_{i}v_j d x.
\end{eqnarray*}
On the other hand, it follows from Proposition (\ref{prop1.1}) that ${(T_{k-1}(D^2 v))^{ij}}_{,i}=0$ and also $$(T_{k-1}(D^2 v))^{ij}v_{ji}=k\s_k(D^2 v).$$ Therefore, we have
\begin{eqnarray*}
m_k&=&\ds \frac{ k! (n-k)!}{(n-1)!\; \omega_{n-1}}\int_{\mathbb{R}^n} v^{\alpha}\s_k(D^2 v)d x\\
&&+\ds \frac{(k-1)! (n-k)!\alpha }{(n-1)!\; \omega_{n-1}}\int_{\mathbb{R}^n} v^{\alpha-1}(T_{k-1}(D^2 v))^{ij}v_{i}v_j d x.
\end{eqnarray*}
We will try to  write the integral of the right hand   in terms of $\s_i(D^2 v)$ and $|\nabla v|^{2i}$, then in terms of $\s_i(A)$ and $|\nabla v|^{2i}$ for $0\le i\le k$.

Directly from the definition of the Newton tensor, we
know
\[
T_{i}(D^2 v)=\s_i(D^2 v)I-T_{i-1}(D^2 v)D^2v=\s_i(D^2 v)I-D^2 vT_{i-1}(D^2 v).
\]
It follows, together with the partial integration
\[\begin{array}{rl}
&\ds\vs \int_{\mathbb{R}^n} v^{\alpha-1}(T_{k-1}(D^2 v))^{ij}v_{i}v_j d x\\
=&\ds\vs\int_{\mathbb{R}^n} v^{\alpha-1}\s_{k-1}(D^2 v)|\nabla v|^2- \int_{\mathbb{R}^n} v^{\alpha-1}(T_{k-2}(D^2 v))^{il}v_{jl}v_jv_i\\
=&\ds\vs\int_{\mathbb{R}^n} v^{\alpha-1}\s_{k-1}(D^2 v)|\nabla v|^2-\frac12 \int_{\mathbb{R}^n} v^{\alpha-1}(T_{k-2}(D^2 v))^{ij}(|\nabla v|^2)_jv_i\\
=&\ds\vs\int_{\mathbb{R}^n} v^{\alpha-1}\s_{k-1}(D^2 v)|\nabla v|^2+\frac{\alpha-1}2 \int_{\mathbb{R}^n} v^{\alpha-2}(T_{k-2}(D^2 v))^{ij}|\nabla v|^2v_iv_j\\
&\ds\vs+\frac{1}2 \int_{\mathbb{R}^n} v^{\alpha-1}{(T_{k-2}(D^2 v))^{ij}}_{,j}|\nabla v|^2v_i+\frac{1}2 \int_{\mathbb{R}^n} v^{\alpha-1}{(T_{k-2}(D^2 v))^{ij}}|\nabla v|^2v_{ij}\\
=&\ds\frac{k+1}{2}\int_{\mathbb{R}^n} v^{\alpha-1}\s_{k-1}(D^2 v)|\nabla v|^2+\frac{\alpha-1}2 \int_{\mathbb{R}^n} v^{\alpha-2}(T_{k-2}(D^2 v))^{ij}|\nabla v|^2v_iv_j.
\end{array}\]
More generally, we have the following claim.\\

{\bf Claim.} For all $1\le l\le k-2$, we have
\beq
\label{relation1}
\begin{array}{llll}
&\ds \vs\int_{\mathbb{R}^n} v^{\alpha-1-l}(T_{k-1-l}(D^2 v))^{ij}|\nabla v|^{2l}v_iv_j\\
 =&\ds\vs\frac{k+l+1}{2(l+1)}\int_{\mathbb{R}^n} v^{\alpha-1-l}\s_{k-1-l}(D^2 v)|\nabla v|^{2(l+1)}\\
 &\ds+\frac{\alpha-l-1}{2(l+1)}\int_{\mathbb{R}^n} v^{\alpha-2-l}(T_{k-2-l}(D^2 v))^{ij}|\nabla v|^{2(l+1)}v_iv_j.
\end{array}
\eeq
As above we have
\beq
\label{eqexp}
\begin{array}{lll}
\ds \vs \int_{\mathbb{R}^n} v^{\alpha-1-l}(T_{k-1-l}(D^2 v))^{ij}|\nabla v|^{2l} v_{i}v_j d x\\
=\ds\vs \int_{\mathbb{R}^n} v^{\alpha-1-l}\s_{k-1-l}(D^2 v)|\nabla v|^{2(l+1)}- \int_{\mathbb{R}^n} v^{\alpha-1-l}(T_{k-2-l}(D^2 v))^{ij}|\nabla v|^{2l}v_{im}v_mv_i\\
=\ds\int_{\mathbb{R}^n} v^{\alpha-1-l}\s_{k-1-l}(D^2 v)|\nabla v|^{2(l+1)}-\frac12 \int_{\mathbb{R}^n} v^{\alpha-1-l}|\nabla v|^{2l}(T_{k-2-l}(D^2 v))^{ij}(|\nabla v|^2)_jv_i.\\
\end{array}
\eeq
On the other hand, we have
\[\begin{array}{rl}
&\ds\vs  -\frac12 \int_{\mathbb{R}^n} v^{\alpha-1-l}|\nabla v|^{2l}(T_{k-2-l}(D^2 v))^{ij}(|\nabla v|^2)_jv_i\\
=&\ds\vs \frac{\alpha-1-l}2 \int_{\mathbb{R}^n} v^{\alpha-2-l}(T_{k-2-l}(D^2 v))^{ij}|\nabla v|^{2(l+1)}v_iv_j\\
&\ds\vs +\frac{k-1-l}2 \int_{\mathbb{R}^n} v^{\alpha-1-l}\s_{k-1-l}(D^2 v)|\nabla v|^{2(l+1)}\\
&\ds
+\frac{l}2 \int_{\mathbb{R}^n} v^{\alpha-1-l}{(T_{k-2-l}(D^2 v))^{ij}}|\nabla v|^{2l}(|\nabla v|^2)_jv_i,
\end{array}\]
which implies
\[\begin{array}{rl}
&\ds\vs  -\frac12 \int_{\mathbb{R}^n} v^{\alpha-1-l}|\nabla v|^{2l}(T_{k-2-l}(D^2 v))^{ij}(|\nabla v|^2)_jv_i\\
=&\ds\vs\frac{\alpha-1-l}{2(l+1)} \int_{\mathbb{R}^n} v^{\alpha-2-l}(T_{k-2-l}(D^2 v))^{ij}|\nabla v|^{2(l+1)}v_iv_j\\
&\ds+\frac{k-1-l} {2(l+1)}\int_{\mathbb{R}^n} v^{\alpha-1-l}\s_{k-1-l}(D^2 v)|\nabla v|^{2(l+1)}.
\end{array}\]
Going back to (\ref{eqexp}), the desired claim yields.
Hence, we have
\[ \begin{array}{rl}
&\ds\vs \int_{\mathbb{R}^n} v^{\alpha-1}(T_{k-1}(D^2 v))^{ij}v_{i}v_j d x\\
=&\ds\vs\frac{k+1}{2}\int_{\mathbb{R}^n} v^{\alpha-1}\s_{k-1}(D^2 v)|\nabla v|^2+\frac{(\alpha-1)\cdots(\alpha-k+1)}{2^{k-1}(k-1)!}\int_{\mathbb{R}^n} v^{\alpha-k}|\nabla v|^{2k}\\
&+\ds\sum_{l=2}^{k-1}\frac{(\alpha-1)\cdots(\alpha-l+1)(k+l)}{2^{l}l!}\int_{\mathbb{R}^n} v^{\alpha-l}|\nabla v|^{2l}\s_{k-l}(D^2 v).
\end{array}\]
Finally, we infer
\beq
\label{massexp3}
\begin{array}{llllll}
&\ds \frac{(n-1)!\; \omega_{n-1}} { (k-1)! (n-k)!}m_k\\
=&\ds k\int_{\mathbb{R}^n} v^{\alpha}\s_k(D^2 v)d x+\frac{(k+1)\alpha}{2}\int_{\mathbb{R}^n} v^{\alpha-1}\s_{k-1}(D^2 v)|\nabla v|^2\\
&\ds+\frac{\alpha (\alpha-1)\cdots(\alpha-k+1)}{2^{k-1}(k-1)!}\int_{\mathbb{R}^n} v^{\alpha-k}|\nabla v|^{2k}\\
&+\ds\sum_{l=2}^{k-1}\frac{\alpha(\alpha-1)\cdots(\alpha-l+1)(k+l)}{2^{l}l!}\int_{\mathbb{R}^n} v^{\alpha-l}|\nabla v|^{2l}\s_{k-l}(D^2 v).
\end{array}
\eeq
Now we want to write $m_k$ in terms of $\s_l(A)$ and $|\nabla v|^{2l}$. Recall
\[
D^2v=vA+\frac{|\nabla v|^2I}{2v},
\]
so that for all $1\le l \le k$ we have
\[
\s_l(D^2v)=v^l\s_l(A+\frac{|\nabla v|^2I}{2v^2})=v^l\sum_{j=0}^l C_{n-j}^{l-j}\s_j(A)\left(\frac{|\nabla v|^2}{2v^2}\right)^{l-j},
\]
where $C_{n-j}^{k-j}=\frac{(n-j)!}{(n-k)!(k-j)!}$. From (\ref{massexp3}), we deduce
\[
\begin{array}{rl}
&\ds \frac{(n-1)!\; \omega_{n-1}} { (k-1)! (n-k)!}m_k\\
=&\ds k\int_{\mathbb{R}^n} v^{\alpha+k}\sum_{j=0}^kC_{n-j}^{k-j}\s_j(A)\left(\frac{|\nabla v|^2}{2v^2}\right)^{k-j}\\
&\ds +(k+1)\alpha\int_{\mathbb{R}^n} v^{\alpha+k}\sum_{j=0}^{k-1}C_{n-j}^{k-1-j}\s_j(A)\left(\frac{|\nabla v|^2}{2v^2}\right)^{k-j}\\
&\ds+\frac{2\alpha (\alpha-1)\cdots(\alpha-k+1)}{(k-1)!}\int_{\mathbb{R}^n} v^{\alpha+k}\left(\frac{|\nabla v|^2}{2v^2}\right)^{k}\\
&+\ds\sum_{l=2}^{k-1}\sum_{j=0}^{k-l}\frac{\alpha(\alpha-1)\cdots(\alpha-l+1)(k+l)}{l!}\int_{\mathbb{R}^n} v^{\alpha+k}C_{n-j}^{k-l-j}\s_j(A)\left(\frac{|\nabla v|^2}{2v^2}\right)^{k-j}\\
=&\ds\int_{\mathbb{R}^n} v^{\alpha+k}\sum_{j=0}^kP_{k-j}(\alpha) \s_j(A)\left(\frac{|\nabla v|^2}{2v^2}\right)^{k-j}.
\end{array}\]
Here for all $0\le j\le k$, $P_j(\alpha)$ is a polynomial of degree $j$ in $\alpha$ with a leading coefficient equal to $k$ when $j=0$, to $k+1$ when $j=1$, to $\frac{2k-j}{(k-j)!}$ when $2\le j\le k-1$ and  to $\frac{2}{(k-1)!}$ when $j=k$. Therefore, we can choose sufficiently negative number $\alpha<0$ such that $(-1)^j P_j(\alpha)>0$ for all $0\le j\le k$. By the assumptions $(-1)^jL_j\ge 0$ for all $1\le j\le k$, which are equivalent to  $(-1)^j\s_j(A)\ge 0$, we have
\[P_{k-j}(\alpha) \s_j(A)=(-1)^{k-j}P_{k-j}(-1)^{j}\s_j(A)\ge 0,\]
i.e.,
 each term on the right hand side in the last inequality is non-negative. This gives $m_k\ge 0$.
 Here we need that $k$ is even.
 Moreover,  if $m_k=0$, we have $\nabla v\equiv 0$, and hence
  $v$ is a constant $1$
 and $\cM$ is the standard euclidean space. We finish the proof.
\qed

\section{Penrose type inequality}
\noindent Let $(\cM^n,g)=(\R^n\setminus \Omega, e^{-2u}\delta)$ be now a  CF manifold,
 where $\Omega$ is a bounded domain such that each connected component of $\Omega$  is  star-shaped such that the second fundamental form  of the boundary $\p \Omega$ is in the cone
 $\Gamma_{k-1}^+(\p\Omega)$.
As before, we assume $2k<n$, $g\in \Gamma_k$, $L_k$ integrable and $u$ satisfies the decay condition at the infinity
 \[
|u|+|x||\nabla u|+|x|^2|\nabla^2 u|=O(|x|^{-\tau}),
\]
with $\tau>\frac{n-2k}{k+1}$. First, we assume $\Omega$ has just one connected component.

\begin{theo}\label{thm}
 Let $(\cM,g)=(\R^n\setminus \Omega, e^{-2u}\delta)$
satisfy the above assumptions.
 Assume, in addition, that $\partial \cM$ is a horizon on $(\cM,g)$ (i.e. $\partial \cM=\partial \Omega \subset \cM$ is minimal)  and $u$  is  constant on $\partial\Omega$.
  Then we have the  following Penrose type inequality
\beq
\label{penrose}
\begin{array}{llll}
m_k
&\ge \ds & \ds \vs \frac{(n-2k)!}{2^k(n-1)!\;\omega_{n-1}} \int_{\cM} e^{(n-2k)u}L_k(g)dvol_g\\
&&\ds\vs + \frac{n-2k}{2^k}\int_{\cM} e^{(n-2k)u}|\nabla u|_g^{2k}dvol_g+\left(\frac{|\p\Omega|}{\omega_{n-1}}\right)^{\frac{n-2k}{n-1}}\\
&\ge&\ds \left(\frac{|\p\Omega|}{\omega_{n-1}}\right)^{\frac{n-2k}{n-1}}.
\end{array}
\eeq
Moreover, if we assume the second fundamental form of $\p\Omega$ is in the cone $\Gamma_{2k-1}$ $ (k\geq 2)$, we have
\begin{equation}
\label{penrose1}
\begin{array}{lll}
m_k &\ge \ds & \ds\vs \frac{(n-2k)!}{2^k(n-1)!\; \omega_{n-1}} \int_M e^{(n-2k)u}L_k(g)dvol_g\\
&&\ds+ \frac{n-2k}{2^k}\int_{\cM} e^{(n-2k)u}|\nabla u|_g^{2k}dvol_g+  \left(\frac{\int_{\p\Omega}R}{(n-1)(n-2)\omega_{n-1}}\right)^{\frac{n-2k}{n-3}}\\
&\ge & \ds \left(\frac{\int_{\p\Omega}R}{(n-1)(n-2)\omega_{n-1}}\right)^{\frac{n-2k}{n-3}}.
\end{array}
\end{equation}
Here $R$ is the scalar curvature of $\p\Omega$ as a hypersurface in  $\R^n$.
\end{theo}

\begin{proof}

Applying Proposition \ref{prop1.1} and Green's formula, we obtain
\begin{eqnarray}
\int_{S_r}(T_{k-1}(D^2 u))^{ij}u_j\nu_i d S-\int_{\p\Omega}(T_{k-1}(D^2 u))^{ij}u_j\nu_i d S=k\int_{B_r\setminus \Omega}\s_k(D^2 u)dx,
\end{eqnarray}
for large $r>0$.
The argument given in the proof of Theorem \ref{mainthm1}, together with (\ref{expan})  to (\ref{eq_b}), implies
\beq
\label{massk}
\begin{array}{lllll}
m_k
&\ge &\ds \vs\vs \frac{(n-2k)!}{2^k(n-1)! \omega_{n-1}} \int_{\cM} e^{(n-2k)u}L_k(g)dvol_g\\
\vspace{2mm}
&&\ds+ \frac{n-2k}{2^k}\int_\cM e^{(n-2k)u}|\nabla u|_g^{2k}dvol_g\\
\vspace{2mm}
&&\ds+\frac{(k-1)!(n-2k)!}{(n-1)!\;\omega_{n-1}} \int_{\p \Omega}(T_{k-1}(D^2 u))^{ij}u_j\nu_i d S.
\end{array}
\eeq
Recall $\nu$ is the normal vector pointing to the infinity. Since  $\partial \cM$ is a horizon of $\cM$, the mean curvature of $\partial \cM$ is equal to zero at the boundary. We denote $H$ the mean curvature of
$\partial \Omega$ in $\R^n$. As $g$ is a conformal metric, the mean curvature of $ \partial \cM$ is equal to $e^u(H-(n-1) \langle \nabla u,\nu\rangle)$. Therefore, on the boundary $\partial \Omega$ we have
\beq
\label{eq3.1}
H-(n-1) \langle \nabla u,\nu\rangle=0.
\eeq
In particular, $\langle \nabla u,\nu\rangle>0$ on the boundary, since we assume the second fundamental form $L$ is in the cone $\Gamma_{k-1}^+(\p\Omega)$.
 On the other hand, from the non-negativity of the scalar curvature, we  have
 \[
\Delta u\ge 0.
\]
Hence, by the Maximum principle, we deduce $u\le 0$ in $\Omega$.  For all $x\in\p\Omega$, we split $T_x\R^n=T_x\p\Omega\oplus \R\nu$ as the sum of tangential part
 and normal part. Let  $e_\beta$  ($1\le \beta\le n-1$)  a
basis of $\partial\Omega$ and $e_n=\nu$.  And Let $B=(D^2u(e_i,e_j))_{1\le i,j\le n}$ be the  Hessian matrix and $B'=(D^2u(e_\alpha,e_\beta))_{1\le \alpha,\beta\le n-1}$ the first $(n-1)\times (n-1)$ block in $B$. Recall that  $u$ is a constant on the boundary $\partial\Omega$. We have
for all $1\le \alpha,\beta\le n-1$
\beq
\label{eq3.2}
D^2 u(e_\alpha,e_\beta)=\langle \nabla u,\nu\rangle L(e_\alpha,e_\beta),
\eeq
where $L$ is the second fundamental form with respect to the normal vector $-\nu$.  Hence, we can compute
\beq
\label{eq3.3}
(T_{k-1}(D^2 u))^{ij}u_j\nu_i=\langle \nabla u,\nu\rangle\frac{\p\s_k(B)}{\p b_{nn}}=\langle \nabla u,\nu\rangle \s_{k-1}(B').
\eeq
Here we have used  the fact $\nabla_\beta u=0$ on the boundary. Gathering (\ref{eq3.1}) to (\ref{eq3.3}), we deduce
\beq
\label{eq3.3bis}
(T_{k-1}(D^2 u))^{ij}u_j\nu_i=\langle \nabla u,\nu\rangle^k \s_{k-1}(L)=\frac{1}{(n-1)^k}\s_1(L)^k\s_{k-1}(L).
\eeq
Recall that in the Garding cone $\Gamma_m^+$, we have the  Newton-MacLaurin inequalities,
\begin{eqnarray}\label{N-M1}
\frac{\sigma_{m-1}\sigma_{m+1}}{\sigma_m^2}&\leq&\frac{m(n-m-1)}{(m+1)(n-m)},\\
\frac{\sigma_1\sigma_{m-1}}{\sigma_m}&\geq&\frac{m(n-1)}{n-m}.\label{N-M2}
\end{eqnarray}
We have $$T_{k-1}(D^2 u))^{ij}u_j\nu_i\geq\left( \frac{(k-1)!}{(n-1)\cdots(n-k+1)}\right)^{\frac{k}{k-1}}\s_{k-1}(L)^{\frac{2k-1}{k-1}}.$$

 From the H\"older inequality and the Aleksandrov-Fenchel inequality (see \cite{Sch}, \cite{GL} and \cite{CW} for example), we have
\begin{eqnarray*}
 \int_{\partial\Omega}(T_{k-1}(D^2 u))^{ij}u_j\nu_i dS &\ge&\left( \frac{(k-1)!}{(n-1)\cdots(n-k+1)}\right)^{\frac{k}{k-1}}\int_{\partial\Omega}\s_{k-1}(L)^{\frac{2k-1}{k-1}}\\
 &\ge&\left( \frac{(k-1)!}{(n-1)\cdots(n-k+1)}\right)^{\frac{k}{k-1}}\left(\int_{\partial\Omega}\s_{k-1}(L)\right)^{\frac{2k-1}{k-1}} |\partial \Omega|^{\frac{-k}{k-1}}\\
&\ge&  \frac{(n-1)!}{(k-1)!(n-k)!} \omega_{n-1}^{\frac{2k-1}{n-1}}|\partial \Omega|^{\frac{n-2k}{n-1}}.
\end{eqnarray*}
Going back to (\ref{massk}), we get the desired inequality (\ref{penrose}).  Now, assume $L\in \Gamma_{2k-1}$, it follows from the Newton-MacLaurin inequality that
\begin{eqnarray*}
\frac{1}{(n-1)^k}\s_1(L)^k\s_{k-1}(L)\ge \frac{(2k-1)!(n-2k)!}{(k-1)!(n-k)!}\s_{2k-1}(L).
\end{eqnarray*}
Hence, again by the Aleksandrov-Fenchel inequality, we get
\begin{eqnarray*}
 \int_{\partial\Omega}(T_{k-1}(D^2 u))^{ij}u_j\nu_idS &\ge& \frac{(2k-1)!(n-2k)!}{(k-1)!(n-k)!} \int_{\partial\Omega}\s_{2k-1}(L)\\
&\ge&  \frac{(n-1)!}{(k-1)!(n-k)!}\omega_{n-1}^\frac{2k-3}{n-3}\left(\int_{\p \Omega} \frac{2\s_2(L)}{(n-1)(n-2)}   \right)^{\frac{n-2k}{n-3}}.
\end{eqnarray*}
In view of (\ref{massk}), we prove inequality  (\ref{penrose1})
and finish the proof. \end{proof}

\begin{rema}
In (\ref{penrose1}), the scalar curvature $R$ could be replaced by other high order  curvature tensor of order small than $k$ which establishes a relationship between the mass $m_{GBC}$ and more geometric objects.
\end{rema}

\begin{rema} We remark that  when $k=1$, our mass $m_1=m_{ADM}$. In this case the Penrose inequality in Theorem \ref{thm} is
\[m_1\ge \left(\frac {|\partial \Omega|}{\omega_{n-1}}\right)^{\frac {n-2}{n-1}},\]
which was already proved in \cite{J}. In fact, our Penrose inequality is motivated by his work.
Note that  we have taken a different test function comparing with the paper \cite{J}.
\end{rema}

Let $\Omega_i$ be the components of $\Omega$, $i=1,\cdots l$, and let $\Sigma_i=\partial \Omega_i$. If we assume that each $\Sigma_i$ is a horizon, we have the following
\begin{coro} With the same condition of Theorem \ref{thm}, and the additional condition that each $\Sigma_i$ is a horizon
  Then we have the the following Penrose type inequality
\begin{eqnarray*}
m_k
&\ge  &  \frac{(n-2k)!}{2^k(n-1)!\;\omega_{n-1}} \int_{\cM} e^{(n-2k)u}L_k(g)dvol_g\\
&& + \frac{n-2k}{2^k}\int_{\cM} e^{(n-2k)u}|\nabla u|_g^{2k}dvol_g+\sum_{i=1}^l \left(\frac{|\Sigma_i|}{\omega_{n-1}}\right)^{\frac{n-2k}{n-1}}\\
& \ge & \sum_{i=1}^l  \left(\frac{|\Sigma_i|}{\omega_{n-1}}\right)^{\frac{n-2k}{n-1}}
\ge \left(\frac{\sum_{i=1}^l|\Sigma_i|}{\omega_{n-1}}\right)^{\frac{n-2k}{n-1}}.
\end{eqnarray*}
Moreover, if we assume the second fundamental form of $\p\Omega$ is in the cone $\Gamma_{2k-1}$ $(k\geq 2)$, we have
\begin{eqnarray*}
m_k
&\ge  &   \frac{(n-2k)!}{2^k(n-1)!\; \omega_{n-1}} \int_\cM e^{(n-2k)u}L_k(g)dvol_g\\
&& + \frac{n-2k}{2^k}\int_{\cM} e^{(n-2k)u}|\nabla u|_g^{2k}dvol_g+\sum_{i=1}^l   \left(\frac{\int_{\Sigma_i}R}{(n-1)(n-2)\;\omega_{n-1}}\right)^{\frac{n-2k}{n-3}}\\
&\ge&  \sum_{i=1}^l   \left(\frac{\int_{\Sigma_i}R}{(n-1)(n-2)\omega_{n-1}}\right)^{\frac{n-2k}{n-3}}
\ge   \left(\frac{\sum_{i=1}^l \int_{  \Sigma_i}R}{(n-1)(n-2)\omega_{n-1}}\right)^{\frac{n-2k}{n-3}}.
\end{eqnarray*}
Here $R$ is the scalar curvature of $\p\Omega$ as a hypersurface in  $\R^n$.
\end{coro}

\begin{exam}\label{Schwarzchild}
$(\cM^n=I\times {\mathbb S}^{n-1},g)$ with coordinates $(\rho,\theta),$ general Schwardschild metrics are given
$$g^k_{\rm Sch}=(1-\frac{2m}{\rho^{\frac nk-2}})^{-1}d\rho^2+\rho^2d\Theta^2,$$
where $d\Theta^2$ is the round metric in ${\mathbb S}^{n-1},$ $m\in \mathbb{R}$ is the ``total mass" of corresponding black hole solutions in the  Lovelock gravity \cite{CTZ, CO}.
When $k=1$ we recover the Schwarzschild solutions of the Einstein gravity.
\end{exam}
Motivated by the Schwarzschild solutions, the above metrics also have the following form of conformally flat which is more convenient for computation (\cite{GWW}).
$$g^k_{\rm Sch}=(1-\frac{2m}{{\rho}^{\frac nk-2}})^{-1}\rho^2+\rho^2d\Theta^2=(1+\frac{m}{2r^{\frac nk-2}})^{\frac{4k}{n-2k}}(dr^2+r^2d\Theta^2).$$

For this metric the Gauss-Bonnet-Chern mass $m_k=m^k$ (one can check it by (\ref{mkspherical}) below) and the black hole (i.e. the horizon) $\Sigma=\partial \Omega=\{r=r_0=(\frac m2)^{\frac k{n-2k}}\} $ and its
area is
$$|\Sigma|= \omega_{n-1}r_0^{n-1},$$
 hence
\begin{eqnarray*}
m_k&=&m^k=(2r_0^{\frac{n-2k}{k}})^k\\
&=&{2^k}\left( \frac{|\Sigma|}{\omega_{n-1}}\right)^{\frac{n-2k}{n-1}}= \frac 1{2^k}\left( \frac{|\Sigma|_{g^k_{\rm Sch}}}{\omega_{n-1}}\right)^{\frac{n-2k}{n-1}}.
\end{eqnarray*}

We remark that the Penrose inequality in Theorem \ref{mainthm3} is not optimal, since in Theorem \ref{mainthm3} the area of $\Sigma$ is computed with
the Euclidean metric  $\delta$, not with the metric $g=e^{-2u}\delta$ itself. In general, if $(\cM^n,g)$ is spherically symmetric, we have the following result.
\begin{prop}\label{spherical symmetric}
Suppose $(\cM^n,g)$ is asymptotically flat CF manifold  with $g=e^{-2u(r)}\delta,$ 
ie., $(\cM^n,g)$ is spherically symmetric, 
then
\begin{equation}\label{mkspherical}
m_k=\lim_{r\rightarrow\infty}\frac{1}{\omega_{n-1}}\int_{S_r}\frac{(u_r)^k}{r^{k-1}}dS_r.
\end{equation}
If $k$ is even, we always have $m_k\geq 0$.
\end{prop}
\begin{proof}
We adopt the equivalent form (\ref{eq_a}) to calculate the Gauss-Bonnet-Chern mass. Denote the radial derivative of $u$ by $u_r\triangleq\frac{\partial u}{\partial r}$. We consider $\Omega=B_r$  being the ball centered at the origin with radius equal to $r$. Thus $\Omega$ can be seen as a level set of $u$ which enable us to use the formulae in the proof of Theorem \ref{thm}.
Let $(e_1,\cdots,e_{n-1})$ be an orthonormal basis of tangent plane on the boundary $\p\Omega$. It follows from (\ref{eq3.2}) that for all $1\le \alpha,\beta\le n-1$, we have
$$
D^2 u(e_\alpha,e_\beta)=\frac{u_r}{r}\delta_{\alpha\beta}
$$
since the second fundamental form on $\p\Omega=S_r$ is equal to $\frac{1}{r}I$ where $I$ is the identity map. By (\ref{eq3.3})
we have
$$T_{k-1}(D^2 u)^{ij}u_i\nu_j
=\frac{(n-1)\cdots(n-k+1)}{(k-1)!r^{k-1}}u_r^k.$$
Going back to (\ref{eq_a}), we get the desired result (\ref{mkspherical}).
\end{proof}

\end{document}